\documentclass{article}
\usepackage{color,latexsym,amsmath,amssymb,amsthm,path,url,enumerate,hyperref, graphicx, tikz}
\usetikzlibrary{calc}
\usepackage{hyperref}
\usepackage{caption}
\usepackage{subcaption}

\usepackage{amssymb,latexsym}
\usepackage{amsmath}
\usepackage{graphicx}
\usepackage{textcomp}

\usepackage{amsthm,amssymb,enumerate,graphicx, tikz}
\usepackage{amscd}
\usepackage{setspace}
\usepackage{comment}
\usepackage{hyperref}
\usepackage{cleveref}

\newtheorem{theorem}{Theorem}[section]

\newtheorem{claim}[theorem]{Claim}
\newtheorem{conjecture}[theorem]{Conjecture}

\newtheorem{definition}[theorem]{Definition}

\newtheorem{problem}[theorem]{Problem}
\newtheorem{question}[theorem]{Question}

\newcommand{\R}{\mathbb{R}}

\newcommand{\F}{\mathcal{F}}

\renewcommand{\geq}{\geqslant}
\renewcommand{\leq}{\leqslant}
\renewcommand{\ge}{\geqslant}
\renewcommand{\le}{\leqslant}

\def\cref#1{Corollary~$\ref{#1}$}

\def\b1{\bar{1}}
\def\cb1{\cdot \bar{1}}
\DeclareMathOperator{\conv}{\mathrm{conv}}
\usepackage{array,color,colortbl}

\newcommand{\M}{\mathcal{M}}

\newcommand{\supp}{\textrm{supp}}

\title{Using the KKM theorem}
\author{
Daniel McGinnis \footnote{Department of Mathematics, Princeton University, USA. \texttt{dm7932@princeton.edu}. Supported by NSF award no. 2402145.}
\and Shira Zerbib  \footnote{Department of Mathematics, Iowa State University, USA. \texttt{zerbib@iastate.edu}. Supported by NSF CAREER award no. 2336239, NSF award no. DMS-195392, and Simons Foundation award no. MP-TSM-00002629.}}
\date{}

\begin{document}

\maketitle

\begin{abstract}
    The KKM theorem, due to Knaster, Kuratowski, and Mazurkiewicz in 1929, is a fundamental result in  fixed-point theory, which has seen numerous extensions and applications. In this paper we survey old and recent generalizations of the KKM theorem and their applications in the areas of piercing numbers, mass partition, fair division, and matching theory. We also give a few new results utilizing KKM-type theorems, and discuss related  open problems.   
\end{abstract}

\section{Introduction}
    Over the last few decades, progress in the field of topological methods in combinatorics has been rapid and has seen the resolution of major open problems. Some examples include Lov\'asz's solution for the Kneser graph coloring conjecture \cite{Lovasz}, Tardos and Kaiser's proof of the $d$-interval piercing number conjecture \cite{tardos, kaiser}, Alon's solution for the necklace splitting problem \cite{alonneck}, and Aharoni's proof of the case $r=3$ of Ryser's conjecture \cite{ryser3}. See \cite{matousek2003using} for more example of topological methods in combinatorics.
In all the aforementioned examples,  the main tool used for the solution  is a classical topological theorem (e.g., Sperner's lemma \cite{sperner} or the Borsuk-Ulam theorem \cite{Borsuk}).

In this survey we focus on applications of yet another classical topological theorem, the KKM theorem \cite{kkm},  and its numerous extensions, to problems  within three areas of research: (1) piercing numbers of families of sets, (2) problems concerning fair division, and (3) mass partition problems.
 Despite their seemingly disparate nature, these subjects share significant commonalities. For instance, fair division problems can be regarded as ``colorful" mass partition problems, and piercing sets with hyperplanes  generate their partition.  However, the  main motif connecting these domains is that they are all amenable to a similar topological approach, utilizing the KKM theorem \cite{kkm} and its extensions.

  \section{The KKM method}\label{sec:framework}

We start by introducing the KKM theorem and several of its prominent extensions. 
Then,  we will demonstrate the method by applying the KKM theorem to prove a classical theorem of Gallai.

  \subsection{KKM-type theorems}\label{sec:kkm}
If $P$  is a polytope with vertex set $\{v_1,\dots, v_k\}$ and $\sigma$ is a face of $P$, we write $i\in \sigma$ if $v_i$ is a vertex in $\sigma$. Let $F(P)$ be the set of all non-empty faces of $P$.

\begin{definition}\hfill
\begin{itemize}
   \item {\em (KKM cover.)} A  {\em KKM cover} of the $(k-1)$-dimensional simplex $\Delta_{k-1}$ is a family  $\{A_1, \dots, A_{k}\}$ of subsets of $\Delta_{k-1}$, all closed or all open, satisfying  the {\em KKM covering condition}:
 for every face  $\sigma$ of  $\Delta_{k-1}$ (including  $\Delta_{k-1}$  itself), we have $\sigma \subseteq \bigcup_{i \subseteq \sigma}A_i.$

\item  {\em (KKMS cover.)} A  {\em KKMS cover} of a polytope $P$ is a family  $\{A_\sigma \mid \sigma \in F(P)\}$ of (possibly empty) subsets of $P$, all closed or all open, satisfying the {\em KKMS covering condition}:
for every face  $\sigma$ of   $P$ (including  $P$  itself), we have $\sigma \subseteq \bigcup_{\tau \subseteq \sigma}A_\tau.$

\item {\em ($(k,n)$-sparse KKMS cover.)} Let $k\le n$ be positive integers, and 
let $P$ be a $(k-1)$-dimensional polytope. A family of sets $\{A^i_\sigma \mid i\in [n],~ \sigma \in F(P)\},$ all closed or all open, is called a {\em $(k,n)$-sparse KKMS cover} of $P$ if for every $I\in \binom{[n]}{n-k+1}$, the family $\{\bigcup_{i\in I}A^i_\sigma \mid \sigma \text{ a non-empty face of }\}$ is a KKMS cover of $P$.
\end{itemize}
\end{definition}

Note that a KKM cover of  $\Delta_{k-1}$  is a KKMS cover of  $\Delta_{k-1}$  where $A_\sigma=\emptyset$ whenever $\dim\sigma \ge 1$. Moreover, if $k=n$ then a $(k,n)$-sparse KKMS cover is a KKMS cover.

In 1929, Knaster, Kuratowski, and Mazurkiewicz \cite{kkm} introduced the notion of KKM covers, and 
proved the following:

\begin{theorem}[The KKM Theorem; Knaster, Kuratowski, and Mazurkiewicz 1929  \cite{kkm}]\label{thm:kkm}
If  $(A_1,\dots,A_{k})$ is a KKM cover of the $(k-1)$-simplex $\Delta_{k-1}$ 
then $\bigcap_{i=1}^{k} A_i \neq \emptyset$. 
\end{theorem}

 Theorem \ref{thm:kkm} has numerous proofs, and it is equivalent to  Brouwer's fixed point theorem and Sperner's lemma, in the sense that all these theorems can be easily proved one from the other. 
 The KKM theorem has inspired many extensions and variants. An important extension
is the colorful  KKM theorem due to Gale~\cite{Gale} that deals with $k$ possibly distinct KKM covers of $\Delta_{k-1}$. % of the $k$-simplex (which we think of as being colored by $k+1$  distinct colors). 
 Let $S_n$ be the group of permutations of the elements in $[n]=\{1,2,\dots,n\}$.

 \begin{theorem}[The Colorful KKM Theorem; Gale 1982 \cite{Gale}]\label{colkkm} 
     For every $j \in [k]$ let $\{A^j_1,\dots,A^j_{k}\}$ be a KKM cover of $\Delta_{k-1}$. Then there exists a permutation  $\pi \in S_{k}$ such that $\bigcap_{i=1}^{k} A_{\pi(i)}^i \neq \emptyset$.
 \end{theorem}

Another well-known generalization is the 
KKMS theorem of Shapley~\cite{shapley}. 
 We say that faces $\sigma_1,\dots,\sigma_{m}$ of $\Delta_{k-1}$  are {\em balanced} if the barycenter of $\Delta_{k-1}$ lies in the convex hull of the barycenters of $\sigma_1,\dots,\sigma_{m}$.
 
 \begin{theorem}[The KKMS theorem, Shapley 1973 \cite{shapley}] \label{kkms}
 Let $\{A_\sigma \mid \sigma \in F(\Delta_{k-1}) \}$ be a KKMS cover of $\Delta_{k-1}$. 
Then there exist  balanced  faces $\sigma_1, \dots, \sigma_{k}$, such that
$\bigcap_{i =1}^{k} A_{\sigma_i}\neq \emptyset.$
\end{theorem}

Note that    it is not longer true that all the sets in a KKMS cover intersect. Instead, Theorem \ref{kkms} ensures that at least $k$ sets intersect, and that the faces corresponding to those sets have a ``nice" structure (they are balanced).

Further generalizations
of the KKMS theorem are a polytopal version due to Komiya~\cite{Komiya} and a colorful KKMS theorem of Shih and Lee~\cite{ShihLee}.
In \cite{FZ} a colorful polytopal KKMS theorem was proved, extending all results above: %For a polytope $P$ containing 0 and a face $\sigma$ of $P$, let $C_\sigma$ be the cone of $\sigma$, emanating from 0.
\begin{theorem}[Frick-Zerbib 2019 \cite{FZ}]
\label{thm:col-komiya}
	Let $P$ be a $(k-1)$-dimensional polytope  and let $p\in P$. 
 Suppose that for every $i\in [k]$, $\{A^i_\sigma \mid \sigma \in F(P) \}$ is a  KKMS cover of $P$. 
 Suppose further that for every $\sigma\in P$ we are given
 $k$ points
	$y^{1}_\sigma, \dots, y^{k}_\sigma \in \sigma$.
	Then there exist faces $\sigma_1, \dots, \sigma_{k}$ of $P$ such that
	$p \in \conv\{y_{\sigma_1}^{1}, \dots, y_{\sigma_{k}}^{k}\}$ and $\bigcap_{i=1}^{k} A^{i}_{\sigma_i}\neq \emptyset.$
\end{theorem}
Here the points $y^i_\sigma$ replace the barycenters in the KKMS theorem.
Recently, Sober\'on \cite{Soberon} proved a ``sparse" generalization of the colorful KKM theorem, where instead of $k$ KKM covers, one considers $n$ families of subsets of $\Delta_{k-1}$ such that the union of any $n-k+1$ of them forms a KKM cover. In \cite{MZ2}, a polytopal generalization of this result was proved:

 \begin{theorem}[McGinnis-Zerbib 2024 \cite{MZ2}]
\label{thm:sparse-komiya}
 Let $n\ge k\ge 2$ be integers. Let $P$ be a $(k-1)$-dimensional polytope with $p\in P$. Assume that for every  $\sigma\in F(P)$, we are given $n$ points $y_\sigma^1,\dots,y_\sigma^n\in \sigma$. If the family  $\{A^i_\sigma \mid i\in [n], \sigma \in F(P)\}$  forms a $(k,n)$-sparse KKMS cover of $P$, then there exists an injection $\pi: [k]\rightarrow [n]$ and faces $\sigma_{1},\dots,\sigma_{k}$ such that $p\in  \conv\{y_{\sigma_{1}}^{\pi(1)},\dots,y_{\sigma_{k}}^{\pi(k)}\}$ and $\bigcap_{i=1}^k A^{\pi(i)}_{\sigma_{i}}\neq \emptyset$.
 \end{theorem}

 In the case $k=n$ we get back Theorem \ref{thm:col-komiya}.
\medskip

The KKM theorem has also a dual variant:

\begin{theorem}[Dual KKM theorem, Sperner \cite{sperner}]\label{dualkkm}
Let $A_1, \ldots ,A_{k}$ be subsets of $\Delta_{k-1}$ that are all closed or all open. Suppose that
\begin{enumerate}
    \item[(a)] $\bigcup_{j=1}^{k} A_j = \Delta_{k-1}$, and
    \item[(b)] for all $x =(x_1,\ldots,x_k) \in \Delta_{k-1}$ and all $j \in [k]$, if $x_j=0$ then $x \in A_j$.
\end{enumerate}
Then $\bigcap_{j=1}^{k} A_j \ne \emptyset$.
\end{theorem}

This dual version was further generalized to product of simplices in \cite{AHZ, NSZ}. 

Finally, KKM-type theorems ``with volume" have also been studied.
Extending a theorem proved 
in \cite{VDPRV} for the unit cube, one obtains the following:

\begin{theorem}[Vander Woude-Zerbib 2023+]\label{thm:VZ}
    Let $P=\Pi_{i=1}^d \Delta_{k_i}$ and let  $\mathcal{C}$ be a KKMS cover of $P$. Then there exists a constant $c$ such that for any $\varepsilon \in (0,1/2]$
there exists a point $p \in P$
such that the $\varepsilon$-ball around $p$ intersects at least $(1 + c\varepsilon)^d$ many sets in $\mathcal{C}$.
\end{theorem}

\begin{problem}
    Prove a colorful version of Theorem \ref{thm:VZ}.
\end{problem}

\subsection{Applying  KKM-type theorems}\label{sec:method}

The topological method utilizing  KKM-type theorems can be described as follows: 
the configuration space of the possible solutions to the problem (e.g., possible 
piercing sets/partitions of goods/mass partitions)
 is  modeled by a polytope $P$. If no ``good" solution is found in the set of  possible solutions $P$, then one obtains a KKM cover, or more generally a KKMS  cover, of  $P$.  The conclusion of the KKM-type theorem, namely that a large enough collection of the sets in the cover intersect, is then translated to a contradiction to the given properties of the hypergraph/goods/mass in question. 
\medskip
    
Let us demonstrate the method by proving a well-known theorem of Gallai. 
Let $\F$ be a family of sets (a hypergraph). A {\em matching} in $\F$ is a subfamily of pairwise disjoint sets. The {\em matching number} $\nu(\F)$ is the largest size of a matching in $\F$. A {\em piercing set} for $\F$ is a set intersecting every set in $\F$. The {\em piercing number} $\tau(\F)$ is the minimum size of a piercing set in $\F$. Clearly, $\nu(\F)\le \tau(\F)$ for every family $\F$. 

The following theorem can be easily proved in an elementary fashion. However, to demonstrate the KKM method, we give a proof that uses the KKM theorem.

\begin{theorem}[Gallai 1959 \cite{gallai}]\label{gallai}
    Let $\F$ be a finite family of compact intervals in $\R$. Then $\tau(\F)=\nu(\F)$. 
\end{theorem}
               \begin{proof}
            Let $\F$ be a family of intervals with $\tau(\F)=k$. We will show  $\nu(F) \geq k$, which entails the theorem.
      Since $\F$ is finite and the intervals in it are compact, by rescaling $\R$ we may assume that all the intervals in $F$ are contained in the open segment $(0,1)$.
\medskip

{\bf Step 1: Modeling the configuration space of possible piercing sets.}
       We  use the  $(k-1)$-simplex  $\Delta_{k-1}$ to model the configuration space of all possible piercing sets of size $k-1$ or less.  Let $\Delta_{k-1}$ be embedded in $\R^{k}$ in the standard way:
             $\Delta_{k-1} = \{(x_1,\dots,x_{k}) \mid x_i\ge 0~ \text{for all }i, \sum x_i=1\}$. 
            Every point $x=(x_1,\dots,x_{k}) \in \Delta_{k-1}$ corresponds to the (piercing) set $P(x)=\{u_1(x),\dots,u_{k-1}(x)\}$ containing  $k-1$ (not necessarily distinct) points, where $u_i(x)= \sum_{j=1}^i x_j \in [0,1]$.

\medskip

{\bf Step 2: Obtaining a KKM cover.} The assumption  $\tau(\F)=k$ implies that 
              $k-1$ or less points do not pierce $\F$, and thus for every $x\in \Delta_{k-1}$, $P(x)$ does not pierce $\F$.  Therefore,  there exists an interval $F\in \F$ that does not contain any of the points in $P(x)$. This implies $F \subset (u_{i-1}(x), u_i(x))$ for some $1 \leq i \leq k$, where we take $u_0(x)=0$ and $u_{k}(x) = 1$ for all $x$.
             
            Define sets $A_1,\dots,A_{k} \subseteq \Delta_{k-1}$ as  follows: $$A_i=\{x\in \Delta_{k-1}\mid \text{~there exists~} F\in \F \text{~such that~} F \subset (u_{i-1}(x),u_i(x))\}.$$ 
            Note that $\Delta_{k-1} \subseteq \bigcup_{i=1}^{k} A_i$.
             
    \begin{claim}
         The family $\{A_1, \dots , A_{k}\}$ forms a KKM cover of $\Delta_{k-1}$.
    \end{claim}
    \begin{proof}
            First,  the sets $A_i$ are open for all $i$ because the intervals in $\F$ are closed. %Indeed, if $F\in \F$ witnesses the fact the $x\in A_i$ (that is $F \subset (u_{i-1}(x),u_i(x))$), then since $F$ is closed, for some small enough $\varepsilon$, every point $x'$ in an $\varepsilon$-neighborhood of $x$ satisfies $F \subset (u_{i-1}(x'),u_i(x'))$, showing $x'\in A_i$.
            We want to show that the KKM covering condition holds. Let $\sigma$ be a face of $\Delta_{k-1}$ and let $x\in \sigma$. If $i\notin \sigma$ then $x_i=0$, and thus $(u_{i-1}(x),u_i(x))=\emptyset$, showing that no $F\in \F$ satisfies  $F \subset (u_{i-1}(x),u_i(x))$. Therefore $x\notin \bigcup_{i\notin \sigma} A_i$. Since $x \in \bigcup_{i=1}^{k} A_i = \big(\bigcup_{i\in \sigma} A_i\big) \cup \big(\bigcup_{i\notin \sigma} A_i\big),$  we must have $x\in \bigcup_{i\in \sigma} A_i$, entailing $\sigma \subseteq \bigcup_{i\in \sigma} A_i$.
 \end{proof}

\medskip

{\bf Step 3: Applying the KKM theorem.}
 By the KKM theorem there exists  $x \in \bigcap_{i=1}^{k}A_i$. Thus for every $i\in [k]$,
since $x\in A_i$,  there exists an interval $F_i \in (u_{i-1}(x),u_i(x))$. Now the set $\{F_1,\dots, F_{k} \}$  is a matching of size $k$, showing $\nu(\F)\ge k$ as desired.
        \end{proof}
        
One of the perks of this proof strategy, is that the KKM theorem can be regarded as a ``black box": one can replace it with a more sophisticated KKM-type theorem, to obtain a corresponding extension of the  result. For example, replacing the KKM theorem with the colorful KKM theorem (Theorem \ref{colkkm}) in the above proof, gives the following colorful extension of Gallai's theorem:  

\begin{theorem}\label{colgallai}
    Let $\F_1,\dots, \F_{k}$ be finite families of compact intervals in $\R$, with $\tau(\F_i) \ge k$ for all $i$. Then there exists a colorful matching, namely a choice of intervals $F_i \in \F_i$, $i=1,\dots,k$, such that $F_1,\dots, F_{k}$ are pairwise disjoint.
    \end{theorem}

    Replacing the KKM theorem in the above proof with the KKMS theorem (Theorem \ref{kkms}) or its polytopal extension due to Komiya \cite{Komiya}, yields the Tardos-Kaiser theorem on piercing $d$-intervals (this proof was first discovered  in \cite{AKZ}). Moreover,  replacing the KKM theorem with Theorem \ref{thm:col-komiya} we obtain a colorful extension of the Tardos-Kaiser result, which was proved by Frick and Zerbib in \cite{FZ} (and of which Theorem \ref{colgallai} is a special case). 
    
    Note that in the  proof,  the large matching for Gallai's theorem was obtained from the fact that the sets $A_i$ corresponding to the vertices of $\Delta_{k-1}$ intersect, together with the fact that the vertices are pairwise disjoint faces of $\Delta_{k-1}$. Since the KKMS theorem and its polytopal extension do not guarantee that sets $A_\sigma$ corresponding to pairwise disjoint faces $\sigma$ intersect, one has to apply some further arguments from matching theory after the application of the KKM-type theorem, to obtain the $d$-interval results. That is, in the KKM method there could be also a fourth step: 
    \medskip
    
    {\bf Step 4: Using matching theory to obtain a large collection of pairwise disjoint faces $\sigma$ whose corresponding sets $A_\sigma$ intersect}.
 \medskip
 
One downside of the KKM method is the following: it seems to work well when the objects we wish to pierce/partition  with (points, lines, hyperplane) are one dimension less then the dimension of the sets or mass in question. One main challenge for further research  is to find ways to apply this  approach in other cases. Two successful attempts can be found in \cite{Z2} and \cite{mcginnis}, where variants of this method where used to bound the piercing numbers of special families of sets in $\R^2$.  
    
The next sections are devoted to applications of the KKM method.

\section{Piercing numbers}
Let $\F$ be a family of sets in $\mathbb{R}^d$. The \textit{piercing number}  $\tau(\F)$ of $\F$ is the size of the smallest set of points in $\mathbb{R}^d$ that intersects every set in $\F$. In a typical piercing problem one wants to find bounds on $\tau(\F)$ for families $\F$  containing sets of a certain type (e.g. convex sets, disks, boxes, etc.),  that have certain given intersection properties. One widely-studied example of an intersection condition that may be imposed on $\F$ is the \textit{$(p,q)$ property}, which is the condition that in every $p$ sets of $\F$, some $q$ sets have a common intersection. A special case of the $(p,q)$ property is the $(p,2)$ property, which is equivalent to the family having matching number $\nu\le p-1$. Thus in many piercing problem one seeks to bound the ratio $\tau/\nu$ for certain families of sets.  
Other type of piercing problems deal with piercing families of sets with lines, hyperplanes, or in general affine subspaces. 

The KKM method has proven  useful in certain piercing problems. In particular, it has been successful to solve problems involving families of sets in $\R$ (see \cite{AKZ, AHZ, Z1, FZ, MZ2}) or line piercing problems in $\R^2$ (\cite{MZ1, Z2}) . In rare occasions the method was pushed forward to bound the (point-)piercing numbers  of sets in $\R^2$ (see \cite{mcginnis, Z2}). It seems hard to generalize the method to higher dimensions. The challenge is twofold:
\begin{itemize}
    \item Modeling the configuration space of all piercing sets (Step 1 in Section \ref{sec:method}) using a polytope that is ``simple enough", so that the KKM-type theorem would give  ``many" sets in the cover that have a common  intersection. 
    \item When piercing with affine spaces of co-dimension at least 2 (in particular, when piercing sets in $\R^d$ with points, for $d\ge 2$), the conclusion of the KKM-type theorem does not provide a partition of the space (as is given in Step 3 in Section \ref{sec:method}), which makes it hard to conclude the existence of a large matching in the family. This problem does not occur when piercing with hyperplanes.
\end{itemize}

Piercing problems is an immense area of research that deserves a survey of its own. We mention here two excellent surveys on related areas: an old survey of Eckhoff on the $(p,q)$-problem \cite{Eckhoff}, and a survey of B\'ar\'any and Kalai on Helly-type problems \cite{barany2022Helly}.  In this manuscript  we  discuss only those  piercing problems that were solved using  the KKM method, or  open problems that may be amenable to the method.

\subsection{Piercing convex sets}
Bounding the piercing numbers of families of convex sets satisfying the $(p, q)$ property
is widely-studied area of research in discrete geometry. It was initiated by a classical theorem of Helly \cite{helly}, asserting that if $\F$ is a family of convex sets in $\R^d$ satisfying the $(d+1,d+1)$ property (that is, every $d+1$ sets in $\F$ intersect) has $\tau(\F)=1$.    
Hadwiger and Debrunner \cite{HD}   proved in 1957 that whenever $p \ge q \ge d + 1$ and
$p(d - 1) < d(q - 1)$, any family of compact, convex sets in $\R^d$ has $\tau \le 
p - q + 1$. They further conjectured  that for every  $p\ge q\ge d+1$ there exists a constant $c=c(p,q;d)$ such that any family of compact, convex sets in $\R^d$ with the $(p,q)$ property has $\tau\le c$. This conjecture  was proved by Alon and Kleitman in 1992 \cite{AK}, and is now commonly known as the $(p,q)$ theorem. The bounds they obtained on the constants $c=c(p,q;d)$ are far from optimal, and finding optimal bounds is known in the literature as the {\em $(p,q)$ problem} (see the survey in \cite{Eckhoff}). The best known general known upper bounds on $c(p,q;d)$ were found by  Keller,  Smorodinsky, and Tardos in \cite{KST}. 

The first case where the optimal bound is not known is $c(4,3;2)$; namely  bounding the piercing numbers of a family of convex sets in $\R^2$ satisfying the property that within any four sets some three intersect. In 2001, Kleitman, Gy\'arf\'as, and T\'oth \cite{432} showed that $3\le c(4,3;2)\leq 13$. Recently, McGinnis \cite{mcginnis} improved the upper bound, showing  $c(4,3;2)\leq 9$. His proof relies heavily on an application of the KKM method, which he used to deduce the existence of two intersecting lines such that each of the resulting quadrants contains the pairwise intersections of intersecting three sets. The remainder of the proof then follows elementary, though highly nontrivial, geometric considerations.  Here McGinnis was able to overcome the obstacle of using the KKM method when piercing sets with co-dimension 2 spaces. However, the analysis needed for the proof was  specifically tailored for the problem and thus not likely to generalize.

Another intriguing question on piercing numbers is the following  conjecture of Dol'nikov,  mentioned in \cite{CMS}:
\begin{conjecture}[Dol'nikov]
Let $K$ be a compact convex set in $\R^2$. Let $\F_1, \F_2, \F_3$ be finite families
of translates of $K$. Suppose that $A \cap B  \neq \emptyset$ for every $A \in \F_i$ and $B \in \F_j$ with $i \neq j$. Then there
exists $j \in [3]$ such that $F_j$ can be pierced by 3 points. 
\end{conjecture}

The conjecture was proved in \cite{CMS} for centrally
symmetric bodies $K$ or when $K$ is a triangle. Further, in \cite{roldan2024colorful} it was proved that the conjecture is true for 4 piercing points instead of 3.

\begin{theorem}[Gomez-Navarro and Rold\'an Pensado 2024 \cite{GR}]
    Let $\F_1,\dots,\F_n$ be finite families of convex sets in $\mathbb{R}^2$. Suppose that $A\cap B\neq \emptyset$ whenever $A\in \F_i$ and $B\in \F_j$ with $i\neq j$. Then one of the following holds:
    \begin{enumerate}
        \item There exists $j\in [n]$ such that $\bigcup_{i\neq j} \F_i$ can be pierced by 1 point.
        \item The family $\bigcup_{i=1}^n \F_i$ can be pierced by 2 lines.
    \end{enumerate}
\end{theorem}

In fact, using the same ideas as above, one may fix a direction $v$ such that one of the lines in the second statement can be chosen to be parallel to this direction.

% \textcolor{red}{Shira: I think the paragraphs below are redundant for this survey. They have nothing to do with KKM, and we talk about axis-parallel rectangles later. What do you think?}

% It is known that such a constant $c(p,q;d)$  does not exist when $q \le d$. For example, when $d=q=2$, a family of pairwise intersecting line segments in general position 
% in the plane has piercing number at least $|\F|/2$, which is not bounded by a constant. However, such a constant may exist if the family is further restricted. For example, Danzer \cite{Danzer} showed that a family of disks in the plane with the $(2,2)$ property has piercing number at most 4.
% Another example is 
% a well-known conjecture of Wegner from 1965 \cite{wegner}, asserting that a family of axis-parallel rectangles in the plane with the $(p+1,2)$ property has $\tau\le 2p$ points (where the best current known upper bound, proven by Correa, Feuilloley,  P\'erez-Lantero, and Soto \cite{soto} is $\tau\le O(p)(\log\log p)^2$). 

% The question of finding or improving bounds on the piercing numbers in general families of convex sets in $\R^d$ with the $(p,q)$ property for $q\ge d+1$, or in restricted families when $q$ may be smaller than $d+1$, has received great attention over the years, see e.g., \cite{Eckhoff, KST,  mcginnis} for general families, and \cite{CSZ, GZ, karasev, KT} for restricted families.

\subsection{Piercing $d$-intervals}

A \textit{$d$-interval} is a union of (not necessarily disjoint and possibly empty) compact intervals on $\mathbb{R}$.   A \textit{separated $d$-interval} is a $d$-interval whose $i$-th interval component lies in $(i,i+1)$ for every $1\leq i\leq d$.

The first application of topology to piercing problem was in a theorem by Tardos  \cite{tardos}, asserting that the piercing number of separated 2-intervals is at most twice their matching number. This extended  Gallai's theorem (Theorem \ref{gallai}) which states that the piercing number of a family of intervals is equal to its matching number. Kaiser then extended this for all $d$:

\begin{theorem}[Tardos 1995 \cite{tardos}, Kaiser 1997 \cite{kaiser}]\label{d-intervals}
If $\F$ is a family of $d$-intervals then $\tau(\F)\leq (d^2-d+1)\nu(\F)$. Further, if $\F$ is a family of separated $d$-intervals then $\tau(\F)\leq (d^2-d)\nu(\F)$.
    
\end{theorem}

The constant $d^2-d+1$ is known to be nearly asymptotically tight, as  Matou\v{s}ek \cite{matousek2001Lower} constructed families of $d$-intervals with  $\tau/\nu = \Omega(\frac{d^2}{\textrm{log}(d)})$. 
Tardos used algebraic topology for his proof. Kaiser's proof is an elegant application of the Borsuk-Ulam Theorem. Later Aharoni, Kaiser, and Zerbib  \cite{AKZ} found another proof of Theorem \ref{d-intervals} using the KKMS theorem and Komiya's polytopal extension of it. This proof followed the KKM method described in Section \ref{sec:method} (and it was the first time where this method was applied). 
Frick and Zerbib \cite{FZ} then used Theorem \ref{thm:col-komiya} to prove a colorful extension of Theorem \ref{d-intervals}, where the families $\F_i$ are thought to have distinct colors, and a {\em rainbow matching} is a collection $\M$ of pairwise disjoint $d$-intervals, such that  $|\M\cap \F_i| \le 1$ for all $i$. 

\begin{theorem}[Frick-Zerbib 2019 \cite{FZ}] \label{coloreddintervals}
  If $\F_i,\dots, \F_{k+1}$ are families of $d$-intervals and  $\tau(\F_i )>k$ for all $i\in[k+1]$, then there exists a rainbow matching of size $\frac{k+1}{d^2-d+1}$. 
Further, if $\F_1,\dots, \F_{kd+1}$ are families of separated $d$-intervals
and  $\tau(\F_i )>kd$ for all $i \in [k+1]$, then there exists a rainbow matching of size $\frac{k+1}{d-1}$.       
\end{theorem}
Taking all the families $\F_i$ to be the same family $\F$, we get back Theorem \ref{d-intervals}. 
 Using Theorem \ref{thm:sparse-komiya}, in \cite{MZ2} a sparse-colorful version of Theorem \ref{d-intervals} was proved. 
No non-topological proofs are known to all the above theorems.

 Theorem \ref{d-intervals} was also extended to families of $d$-intervals having  the $(p,q)$-property:
\begin{theorem}[Zerbib 2019 \cite{Z1}]\label{dpps} If a family $\F$ of separated $d$-intervals  has the $(p,p)$ property, then $\tau(\F) \le d^{\frac{p}{p-1}}$.
Moreover, if a family $\F$ of $d$-intervals has the $(p,q)$ property, then $\tau(H)\le \max\Big\{\frac{2^{\frac{1}{q-1}}(ep)^\frac{q}{q-1}}{q} d^{\frac{q}{q-1}} + d,~ 2p^2d\Big\}.$
\end{theorem}
The method  used for the proof was a combination of the KKM method, and a generalization of an elementary method introduced by Alon \cite{alon1}.  

The KKM method can be used to prove ``dual" results as well.
For example, in \cite{AHZ} this method, applied together with an extension of the dual KKM theorem (Theorem \ref{dualkkm}), was used to bound  the edge-covering number (namely
the minimal number of edges covering the entire vertex set) of a family of $d$-intervals in terms of a
parameter expressing independence of systems of partitions of the $d$ unit
intervals.

\subsection{Piercing axis-parallel rectangles}

Wegner conjectured in 1965 that the piercing number of families of axis parallel rectangles is linearly bounded by their matching number:  

\begin{conjecture}[Wegner 1965 \cite{wegner}]\label{conj:wegner}
If  $\F$ is a finite family of axis-parallel rectangles in $\R^2$ then $\tau(\F)\le 2\nu(\F)$.
\end{conjecture}
It is not even known (and was conjectured by  Gy\'arf\'as and Lehel \cite{GH}) that there exists a constant $c$ such that a family $\F$ meeting the conditions of Conjecture \ref{conj:wegner} satisfies $\tau(\F) \leq c\nu(\F)$.
 The best  known upper bound, proven by Correa, Feuilloley,  P\'erez-Lantero, and Soto \cite{soto} is $\tau \le O(\nu)(\log\log \nu)^2$). A construction of Fon-DerFlaass and Kostochka \cite{flaass1993Covering} demonstrates a family of axis parallel rectangles with $\tau \geq \lfloor 5\nu/3 \rfloor$.

  Wegner's conjecture can be reduced to a (weaker) problem on families of $d$-intervals in the following way. Let $\F$ be a family as in Wegner's conjecture and let $p=\nu(\F)$. Orthogonally projecting $\F$ onto the $y$-axis, by Gallai's theorem there exist $p$ horizontal lines that intersect all the rectangles in $\F$. For a rectangle $F\in \F$, let $\bar{F}$ be the  $p$-interval obtained by intersecting $h$ with those $p$ horizontal lines. Let $\bar{\F} = \{\bar{F}\mid F\in \F\}$. Note that  $\bar{\F}$ has a  restricted structure as a family of $p$-intervals; for example, if $\bar{F},\bar{F'}\in \bar{\F}$ and an interval  component of $\bar{F}$ precedes an interval  component of $\bar{F'}$ on some line, than all the non-empty interval components of $\bar{F}$ precede all the non-empty components of $\bar{F'}$. 
Clearly,  $\tau(\F)\le \tau(\bar{\F})$. Therefore, a constant bound in Wegner's conjecture is implied by
    the following two conjectures:
\begin{conjecture}\label{conjapr1}
There exists a  constant $A$ such that  every finite family of axis-parallel rectangles $\F$ satisfies $\tau(\bar{\F}) \le A\nu(\bar{\F})$. 
\end{conjecture}
\begin{conjecture}\label{conjapr2}
There exists a  constant $B$ such that  every finite family of axis-parallel rectangles $\F$ satisfies $\nu(\bar{\F}) \le B\nu(\F)$.
\end{conjecture}
Conjecture \ref{conjapr1} is interesting by itself, as it may shed light on the behavior of families of axis-parallel rectangles as well as $d$-intervals. It is possibly approachable with  the KKM method. 

For a while it was believed that Wegner's conjecture may be true in higher dimensions as well. This was refuted recently when Tomon \cite{Tomon2023Lower} constructed a family of axis-parallel boxes in $\R^d$, for all $d\ge 3$, with $\tau \ge \Omega_d(\nu) \cdot \left(\frac{\log(\nu)}{\log \log(\nu)} \right)^{d-2}$.

\subsection{Piercing sets with lines}\label{sec:LinePiercing}

We say that a family $\F$ of sets in $\R^d$ is {\em pierced by $k$ lines}, if there exist $k$ lines in $\R^d$ whose union intersects every set in $\F$. The {\em line-piercing number} of $\F$ is the smallest $k$ such that $\F$ is pierced by $k$ lines. We say that $\F$ has the {\em $T(k)$ property} if every $k$ or fewer sets in $\F$ are pierced by a line.

The question of bounding the line-piercing number of families of convex sets in $\R^d$ has received considerable attention.
In 1969 Eckhoff \cite{Eckhoff1} proved that if a family of compact convex sets in $\R^2$  has the $T(k)$ property, for $k\ge 4$, then it is pierced by two lines. In fact, he showed that onethe direction of one of the lines can be fixed in advance. In 1993 
Eckhoff \cite{Eckhoff3}  proved that a family of compact convex sets  with the $T(3)$ property is pierced by 4 lines, and conjectured that  this bound can be improved to 3, which is best possible \cite{Eckhoff2}. This conjecture was proved in \cite{MZ1}, using the KKM-theorem. In the proof, one of the three lines can be chosen to pass through a point fixed in advance. Again, by replacing the KKM theorem with the colorful  KKM theorem, one obtains a colorful versions of these results. We note that we believe the first statement of Theorem \ref{thm:T3} below to be true when there are only 3 families, rather than 6, however, this stronger statement does not seem to follow from a straightforward application of the colorful KKM theorem.
\begin{theorem}[McGinnis-Zerbib 2022 \cite{MZ1}]\label{thm:T3} \hfill
\begin{enumerate}
    \item
    Let $\F_1,\dots,\F_6$ be families of compact convex sets in $\R^2$ such that every colorful collection of  three sets  have a line transversal. Then there exists  $1\leq i\leq 6$ such that the line-piercing number of $\F_i$ is at most $3$.

\item
    Let $\F_1,\dots,\F_4$ be families of compact convex sets in $\R^2$ such that every colorful collection of  four sets  have a line transversal. Then there exists   $1\leq i\leq 4$ such that the line-piercing number of $\F_i$ is at most $2$.
    \end{enumerate}
Moreover, in both statements one can fix in advance a point outside the convex hull of the union of all the sets through which one of the piercing lines passes.    
\end{theorem}

Here the application of the KKM theorem in the proof is somewhat different than before: we use a simplex to model the configuration space of 
 points on a circle, and then we pairs of points to form lines in a certain pre-prescribed way. This  approach has been fruitful in proving results on piercing numbers and line-piercing numbers in $\R^2$, see \cite{Z2, mcginnis2, GR}, and we believe it has potential for further applications. 

We  demonstrate it by providing a colorful version  of Eckhoff's result that a finite family of convex sets in $\mathbb{R}^2$ satisfying the $T(4)$ property can be pierced by 2 lines, where the direction of one of the lines is fixed in advance. 

\begin{theorem}\label{thm:t(4)}
    Let $\F_1,\dots,\F_4$ be finite families of compact convex sets in $\mathbb{R}^2$ such that every colorful collection of four sets can be pierced by a line, and let $v$ be a fixed direction. Then there exists  $1\leq i\leq 4$ such that $\F_i$ is  pierced by two lines, one of which is parallel to $v$.
\end{theorem}
\begin{proof} We apply the KKM method with the colorful KKM theorem.
\medskip

{\bf Step 1: Modeling the configuration space of all relevant two lines.}
By rescaling the plane, we may assume that all the sets are contained in the  disk centered in the origin with radius $1/2$. Let  $C$ be the  circle bounding this disk. Assume without loss of generality that $v$ is in the vertical direction. 

A point $x=(x_1,x_2,x_3,x_{4})\in \Delta_{3}$ corresponds to two lines $\ell_1(x), \ell_2(x)$ as follows. We first obtain that line  $\ell_1(x)$ by taking the vertical line tangent to $C$ at the point $(1/2,0)$ and translating it a distance of $x_1 + x_2$ to the left. Let $D$ and $U$ 
%\textcolor{red}{I think $A$ (above) $B$ (below) are better here. Otherwise you get $U$ (up) and then the counterpart should be $D$ (down), or $L$ (low) and then the counterpart should be $H$ (high). $U$ and $D$ just don't fit together} 
be the intersection points of $\ell_1(x)$ with $C$, where $D$ lies below $U$ (note that $U=D=(1/2,0)$ if $x_1+x_2=0$ and $U=D=(-1/2,0)$ if $x_1+x_2=1$). 

If $x_1+x_2 \neq 0$, let $p(x)$ be the point on $C$ such that the counterclockwise arc length from $D$ to $p(x)$ is an $\frac{x_1}{x_1+x_2}$ proportion of the arc length from $D$ to $U$ counterclockwise. Otherwise, if $x_1+x_2 = 0$, let $p(x)= U=D = (1/2,0)$. 
Similarly, if $x_3+x_4 \neq 0$, let $q(x)$ be the point on $C$ such that the counterclockwise arc length from $U$ to $q(x)$ is an $\frac{x_3}{x_3+x_4}$ proportion of the counterclockwise arc length from $U$ to $D$. Otherwise, if  $x_3+x_4 = 0$, let $q(x)= U=D = (-1/2,0)$. Note that it is not possible that both $x_1+x_2 =0$ and $x_3+x_4 =0$.
Finally, let $\ell_2(x)$ be the line connecting $p(x)$ and $q(x)$. It is possible that $p(x)=q(x)$ in which case they do not define a line. In this case take $\ell_2(x)$ to be a point (see Figure \ref{figure1}).

Now let $R_1(x),\dots,R_4(x)$ be the open quadrants in $D$ defined by the lines $\ell_1(x), \ell_2(x)$, as in Figure \ref{figure1}. Note that by construction, for every $i\in [4]$, if $x_i=0$ then $R_i(x)=\emptyset$.

\medskip
{\bf Step 2: Obtaining KKM covers.}
For $i,j\in [4]$ define  $A^i_j\subset \Delta_{3}$ to be the set of points $x \in \Delta_{3}$ such that $R_j(x)$ contains a set from $\F_i$. Note that $A^i_j$ is open, and if $x_j=0$, then $x\notin A^i_j$ for all $j$. Moreover, if no two lines pierce any $\F_i$, then  $\Delta_{3} \subset \bigcup_{j=1}^{4} A^i_j$ for all $i$. It follows that the conditions of the colorful KKM theorem  (Theorem \ref{thm:kkm}) are satisfied, namely for all $i\in [4]$, the sets  $A^i_1,\dots,A^i_4$ form a KKM cover.

\medskip
{\bf Step 3: Applying the colorful KKM theorem.}
By Theorem \ref{thm:kkm} there exists a permutation $\pi:[4] \rightarrow [4]$ and a point $x_0\in \bigcap_{i=1}^{4} A_{\pi(i)}^i$. Therefore, there is a colorful selection of sets that are contained in distinct quadrants defined by the 2 lines associated to $x_0$. However,  there is no line that can pierce such a collection of sets, contradicting the condition of the theorem.
\end{proof}

\begin{center}
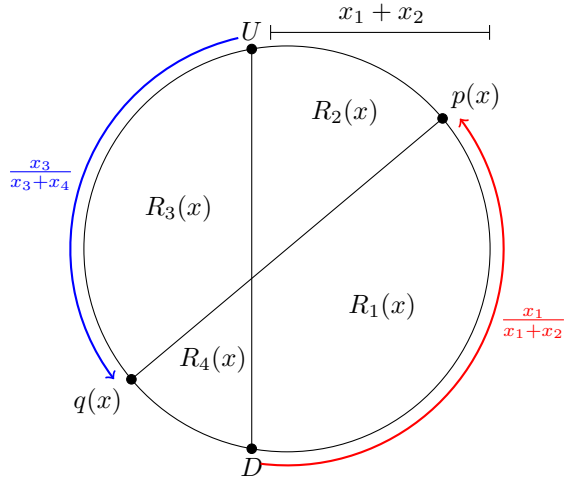
\begin{figure}
\begin{center}
    \begin{tikzpicture}[scale=.9]
    \draw (0,0) circle (3cm);

\filldraw [black] (100:3cm) circle (2pt) node[above] {$U$};

\filldraw [black] (260:3cm) circle (2pt) node[below] {$D$};
\filldraw [black] (40:3cm) circle (2pt) node[above right] {$p(x)$};
\filldraw [black] (220:3cm) circle (2pt) node[below left] {$q(x)$};

\draw (100:3cm) -- (260:3cm);
\draw (40:3cm) -- (220:3cm);
\draw[thick,red,->] ([shift=(263:3.2cm)]0,0) arc  (263:397:3.2cm);
\filldraw [red] (340:3.2cm) node[right] {$\frac{x_1}{x_1 + x_2}$};

\draw[thick,blue,->] ([shift=(103:3.2cm)]0,0) arc  (103:217:3.2cm);
\filldraw [blue] (160:3.2cm) node[left] {$\frac{x_3}{x_3 + x_4}$};

\draw[|-|] (-.25cm,3.2cm) -- (3cm,3.2cm);
\filldraw [] (1.4cm,3.2cm) node[above] {${x_1 + x_2}$};

\filldraw [] (340:1.5cm) node[below] {$R_1(x)$};
\filldraw [] (150:1.85cm) node[below] {$R_3(x)$};
\filldraw [] (70:2.5cm) node[below] {$R_2(x)$};
\filldraw [] (230:1.7cm) node[below] {$R_4(x)$};

\end{tikzpicture}
\end{center}
\caption{In the proof of Theorem \ref{thm:t(4)}, a point $x=(x_1,\dots,x_4)\in \Delta_3$ corresponds to two lines, one of which is vertical,  and four regions. The length of the red arc is an $\frac{x_1}{x_1 + x_2}$ proportion of the length of the arc from $D$ to $U$ counterclockwise. The length of the blue arc is an $\frac{x_3}{x_3 + x_4}$ proportion of the length of the arc from $U$ to $D$ counterclockwise.}
    \label{figure1}
\end{figure}
\end{center}

%It is possible to modify the above arguments to obtain different divisions of the plane. For example, given $x = (x_1,\dots,x_4)\in \Delta_3$, we can again define a vertical line that's translated to the right a distance of $x_1$ to the left (instead of $x_1+x_2$). Let $p(x)$ (respectively $q(x)$) be the points on $C$ such that the counterclockwise arc length from $U$ to $p(x)$ ($q(x)$) is an $\frac{x_2}{x_2+x_3+x_4}$ ($\frac{x_2+x_3}{x_2+x_3+x_4}$) proportion of the arc length from $U$ to $D$ following $C$ counterclockwise (see Figure \ref{figure2}). Define $m(x)$ to be the midpoint of the vertical line, and we take the lines defined by $p(x)$ and $m(x)$ and by $q(x)$ and $m(x)$ (in the case that $x_1=1$, these lines degenerate to a point). We define the regions $R^1_x,R^2_x, R^3_x,R^4_x$ to the be the resulting open regions defined by these lines as in Figure \ref{figure2}.

Another direction of research concerns the notion of {\em fractional line-piercing numbers}. More formally, one would like to answer the following:
    \begin{question}\label{qfraclines}
    What is the largest constant $\alpha(k) \in (0,1)$ such that for any finite family  $\mathcal{F}$ of convex sets in $\R^2$ with the $T(k)$-property, there is a line intersecting $\alpha(k)|\mathcal{F}|$ members of $\mathcal{F}$?
    \end{question}
Katchalski and  Liu \cite{KatchalskiLiu} showed that  $\alpha(k)$ tends to 1 as $k$ goes to infinity. Holmsen \cite{holmsen} proved that $\alpha(k)\leq \frac{k-2}{k-1}$, $\frac{1}{3}\leq \alpha(3)\leq \frac{1}{2}$ and $\frac{1}{2}\leq \alpha(4)\leq \frac{2}{3}$. It is possible that the topological approach described in this proposal can be used for improve the upper bounds on $\alpha(k)$. Indeed, in \cite{AKZ} we used the topological method to bound fractional (point-)piercing numbers of certain families, and in \cite{MZ1,Z2} we used the topological method to bound the (integral) line-piercing numbers of certain families. Combining the two approaches may be useful for Question \ref{qfraclines}. 
It will be also interesting to determine  $\alpha(k)$ when the family $\mathcal{F}$ is further restricted (e.g., when $\F$ consists of disks).

As mentioned above, the KKM method works well for bounding line-piercing numbers in $\R^2$, but it is  an intriguing problem is to apply it in higher dimensions. A motivating conjecture is the following:

\begin{conjecture}[Mart\'inez-Rold\'an-Rubin 2020 \cite{MRR}]
     There exists a constant $c$ with the following property: if $\F$ be a finite family of compact pairwise intersecting convex sets in $\R^3$ then there is a line intersecting $c|\F|$ sets of $\F$. 
\end{conjecture}

 B\'ar\'any \cite{barany} gave a positive answer when the family consists of  cylinders.

% \begin{figure}
% \begin{center}
%     \begin{tikzpicture}[scale=.9]
%     \draw (0,0) circle (3cm);

% \filldraw [black] (100:3cm) circle (2pt) node[above] {$U$};
% \filldraw [black] (260:3cm) circle (2pt) node[below] {$D$};
% \filldraw [black] ($(100:3cm)! 0.5!(260:3cm)$) circle (2pt) node[right] {$m(x)$};
% \filldraw [black] (150:3cm) circle (2pt) node[above left] {$p(x)$};
% \filldraw [black] (220:3cm) circle (2pt) node[below left] {$q(x)$};

% \draw (150:3cm) -- ($(100:3cm)! 0.5!(260:3cm)$);
% \draw (220:3cm) -- ($(100:3cm)! 0.5!(260:3cm)$);
% \draw (100:3cm) -- (260:3cm);

% \draw[thick,red,->] ([shift=(103:3.1cm)]0,0) arc  (103:147:3.1cm);
% \filldraw [red] (120:3.35cm) node[left] {$\frac{x_2}{x_2 + x_3+x_4}$};

% \draw[thick,blue,->] ([shift=(103:3.2cm)]0,0) arc  (103:217:3.2cm);
% \filldraw [blue] (180:3.2cm) node[left] {$\frac{x_2+x_3}{x_2+x_3+x_4}$};

% \draw[|-|] (-.25cm,3.2cm) -- (3cm,3.2cm);
% \filldraw [] (1.4cm,3.2cm) node[above] {$x_1$};

% \filldraw [] (0:1.5cm) node[below] {$R^1_x$};
% \filldraw [] (170:1.85cm) node[below] {$R^3_x$};
% \filldraw [] (120:2.3cm) node[below] {$R^2_x$};
% \filldraw [] (230:1.7cm) node[below] {$R^4_x$};

% \end{tikzpicture}
% \end{center}
% \caption{An example of another way to configure four regions in the plane using $\Delta_3$.}
%     \label{figure2}
% \end{figure}

\subsection{Piercing convex sets with hyperplanes}\label{sec:hyperplanes}

Another type of piercing problems that has received a significant attention concerns piercing families of convex sets with hyperplanes. The main result in this area is due to Alon and Kalai \cite{alon1995bounding}, which serves as an analogous to the aforementioned $(p,q)$ theorem. Specifically, they proved that for every integers $p\ge q \ge d+1$ there is a constant $C=C(p,q;d)$ such that if a finite family $\F$ of convex sets in $\mathbb{R}^d$ satisfies the property that for every $p$ sets in the family, some $q$ of them can be pierced by a hyperplane, then there are $C$ hyperplanes whose union intersects each set in $\F$. 

Here we describe a general result on hypergraph piercing the was (essentially) proved recently using the KKM method by Sober\'on and Yu \cite{soberon2023}. It concerns 
the following notion of \textit{$\Delta$-space}. 

\begin{definition}[Sober\'on - Yu 2023 \cite{soberon2023}]
    A set  $\mathcal{H}$  of partitions of $\mathbb{R}^d$ into $n$ parts is  a \textit{$\Delta$-space} if there exists a map $R:\Delta_{n-1} \rightarrow \mathcal{H}$, $x\mapsto (C_1(x),\dots,C_n(x))$, such that for every $x=(x_1,\dots,x_n)\in \Delta_{n-1}$ and $i\in [n]$ the following holds:
    \begin{enumerate}
        \item if  $x_i = 0$ then  $C_i(x)$ has Lebesgue measure zero, and
        \item if $\mu$ is a finite measure absolutely continuous with respect to the Lebesgue measure,  then $x \mapsto \mu(C_i(x))$ is a continuous map from $\Delta_{n-1}$ to $\mathbb{R}$.
    \end{enumerate}
\end{definition}

A $\Delta$-space is a general description of a space to which that one can apply the KKM method. For example, the partitions of $\mathbb{R}$ and $\mathbb{R}^2$, in the proofs of these Theorems \ref{gallai} and \ref{thm:t(4)} respectively,  are $\Delta$-spaces.

A natural class of $\Delta$-spaces, which was introduced in \cite{soberon2023}, is the class of  {\em spaces of nested hyperplane partitions}, which we describe here. First, given a nonzero vector $v\in \mathbb{R}^d$ and $\alpha\in [0,1]$, define  two halfspaces
\begin{align*}
    &H^+(\alpha,v) = \left\{ y\in \mathbb{R}^d \mid \langle y, v \rangle \geq \frac{2\alpha -1}{1- |2\alpha -1|} \right\}\\
    &H^-(\alpha,v) = \left\{ y\in \mathbb{R}^d \mid \langle y, v \rangle \leq \frac{2\alpha -1}{1- |2\alpha -1|} \right\}.
\end{align*}
When $\alpha = 0$ and $\alpha = 1$, we take $H^+(\alpha,v) = \emptyset$ and $H^+(\alpha,v) = \mathbb{R}^d$, respectively, and vice versa for $H^-(\alpha,v)$.

Now, for integers $d\ge 1, n\ge 2$ let $v_1,\dots,v_{n-1}\in \R^d$ be vectors fixed in advance. The nested hyperplane partition of $\R^d$ corresponding to $v_1,\dots,v_{n-1}$ is modeled by  $\Delta_{n-1}$ as follows. 
For $n=2$, each point  $x=(x_1,x_2) \in \Delta_{1}$ corresponds to the convex partition $(C_1(x),C_2(x))$  of $\mathbb{R}^d$, where $C_1(x) = H^+(x_1,v)$ and $C_2(x) = H^-(x_1,v)$.
Suppose now $n\geq 3$, and let $x = (x_1,\dots,x_n) \in \Delta_{n-1}$. If $x_n=1$ then $x$ corresponds to the partition $(\emptyset,\dots,\emptyset, \mathbb{R}^d)$. Otherwise,  consider the  partition $(C_1'(x'),\dots,C_{n-1}'(x'))$ obtained from the vectors $v_1,\dots,v_{n-2}$ and the point $x'=\frac{(x_1,\dots,x_{n-1})}{\sum_{i=1}^{n-1} x_i} \in \Delta_{n-2}$. We now define the partition $(C_1(x),\dots,C_n(x))$ corresponding to $x$ by $C_i(x) = C_i'(x') \cap H^-(x_n,v_{n-1})$ for $i\leq n-1$ and $C_n = H^+(x_n,v_{n-1})$. Note that $C_i = \emptyset $ when $x_i = 0$.

Sober\'on and Yu \cite{soberon2023} used the KKM method to prove a fair division result concerning nested hyperplane partitions. Their proof ideas can be adopted to prove the following hypergraph piercing theorem.

\begin{theorem}
    Let $\F$ be a finite family of convex sets in $\mathbb{R}^d$. Let $v_1,\dots,v_{n-1}$ be vectors in $\R^d$, and for every $x\in \Delta_{n-1}$ let $(C_1(x),\dots,C_n(x))$ the corresponding nested hyperplane partition.
    If for every point $x\in \Delta_{n-1}$, there is  $i\in [n]$ such that the interior of $C_i(x)$   does not contain a set in $\F$, then there are $n-1$ hyperplanes orthogonal to $v_1,\dots,v_{n-1}$ whose union intersects every set in $\F$.
\end{theorem}

\subsection{Other piercing problems}
In traditional piercing problems  the dimension of the piercing spaces in the condition of the theorem is the same as the dimension of the piercing spaces in the conclusion of the theorem. For example, in Theorem \ref{d-intervals} the condition is that  every $\nu+1$ sets in the family are pierced by a point, and the conclusion is that there are $(d^2-d+1)\nu$ points piercing the entire family. Similarly, in Theorem \ref{thm:T3} the condition is that  every $3$ sets in the family are pierced by a line, and the conclusion is that the family is pierced by $3$ lines. However, one can also mix and match. 

One interesting question in this direction is how many hyperplanes are needed to pierce a family of convex sets in $\R^d$ with the $(p,q)$ property.
If $\F$ is a finite family of convex sets in $\R^d$ with the $(p+1,2)$ property, then $\F$ is pierced by $p$ hyperplanes; this follows from projecting $\F$ onto a line and using Theorem \ref{gallai}. However, this bound is far from being tight. In \cite{Z2}  the KKM method was used to improve this bound: 

\begin{theorem}[Zerbib 2024 \cite{Z2}]\label{main2}
For every $d\ge 2$, if $\F$ is a finite family of compact convex sets in $\R^d$ with the $(p+1,2)$ property then $\F$ is pierced by $\lfloor \frac{p}{2} \rfloor+1$ hyperplanes.  Moreover, when $d=2$ this bound is tight.
\end{theorem}

The bound in Theorem \ref{main2} is tight for $d=2$ by taking $\F$ to be  the family of edges of a regular $(2p+1)$-gon. By replacing the KKM theorem with the colorful KKM theorem, a colorful version of this result is proved in \cite{Z2}. We also note that Theorem \ref{main2} is a special case of Conjecture 4.1 posed in \cite{mcginnis2}.

\begin{question}
    Is it true that a finite family of compact convex sets in $\R^d$ with the $(p,2)$ property is pierced by $\frac{p}{d} + O(1)$ hyperplanes? 
\end{question}

In \cite{fricketat} a special case of this question was verified: it was proved there that a family of convex polytopes in $\R^3$ with the $(8,2)$ property can be pierced by 3 hyperplanes.

\section{Mass partition}\label{sec:mass}
Mass partition problems concern the partitions we can induce on a measure, or
a family of measures, or finite sets of points in Euclidean spaces by dividing the
ambient space into pieces, while satisfying some additional geometric
properties. This area of research has many connections with topology, discrete geometry, and computer science. For an excellent survey on mass partition problems and their history and applications, see \cite{RPS}. 

Topological methods have been extensively applied to mass partition problems. A famous example is the ham sandwich theorem, which states that $d$  measures
in $\R^d$ can be simultaneously split into two equal parts each, using a single hyperplane. Another example is necklace splitting theorem of Alon \cite{alonneck}, stating that a necklace with $kn$   beads, with $ka_{i}$ beads of  color $i$,
can be partitioned into $k$ parts
$a_{i}$ beads of color $i$. Both  theorems where proved using the topological Borsuk-Ulam theorem \cite{Borsuk}. 

In this survey we will focus on mass partition results that were proved using the KKM method, or open problems that may be amenable to the method. 
%We  demonstrate the method by proving a theorem that extends   a result of Aurenhammer, Hoffmann, and  Aronov \cite{AHA}, as well as a theorem by Sober\'on and Yu (Theorem 1.3 in \cite{soberon2023}):
We  demonstrate the method by proving the following mass partition theorem. In \cite{soberon2023} Sober\'on and Yu proved the special case $k=2$.  

\begin{theorem}\label{thm:mass}
    Let $\mu_1, \dots, \mu_{2k}$ be absolutely
continuous probability measures in $\R^2$ with compact support, and let $\alpha_1,\dots, \alpha_{2k}$ be positive reals so that $\alpha_1+\dots+ \alpha_{2k} =1$. Then there is a partition of the plane into $2k$ convex regions $Q_1,\dots, Q_{2k}$ using $k$ lines, and such that $Q_1,\dots, Q_{2k}$  are ordered in a counterclockwise order, and a permutation $\pi: [2k] \rightarrow [2k]$, such that $\mu_{\pi(i)}(Q_i) \geq \alpha_i$ for all $i\in [2k]$.
\end{theorem}
\begin{proof}
We apply the KKM method together with colorful KKM theorem. Let $D$ be the unit disk, and let $S^1$ be its boundary. Since the measures have compact support, we may assume that their support is contained in $D$.

\medskip
{\bf Step 1: Modeling the configuration space of relevant partitions.}

Let $f:[0,1] \rightarrow S^1$ be  defined by $f(t)=(\textrm{cos}(2\pi t), \textrm{sin}(2\pi t))$. A point $x=(x_1,\dots,x_{2k})\in \Delta_{2k-1}$ corresponds to $2k$ points $z_1,\dots,z_{2k}$ on $S^1$ given by $z_i(x)=f(\sum_{j=1}^i x_{j})$ for $0\leq i\leq 2k$ (note that $z_{2k}(x)=z_0(x)$). For $0\leq i\leq 2k-1$ let $\ell_i(x)$ be the line segment $[z_i(x),z_{i+k}(x)]$, where addition is modulo $2k$. If $z_i(x)\neq z_{i+k}(x)$, let $\ell'_i(x)$ be the line containing $\ell_i(x)$. 

For $0\leq i\leq 2k-1$, define $H_i^+(x)$ and $H_i^-(x)$ as follows. 
if $z_i(x)\neq z_{i+k}(x)$, let $H_i^+(x)$ be the closed halfspace defined by $\ell'_i(x)$ containing the counterclockwise arc going from $z_i(x)$ to $z_{i+k}(x)$, and let $H_i^-(x)$ the other closed halfspace defined by $\ell'_i(x)$. Otherwise, if $z_i(x)=z_{i+k}(x)$ then either $\sum_{j=i+1}^{i+k} x_j = 0$ or $\sum_{j=i+1}^{i+k} x_j = 1$. If $\sum_{j=i+1}^{i+k} x_j = 0$ let $H_i^+(x) = \emptyset$ and $H_i^-(x) = \mathbb{R}^2$, and if $\sum_{j=i+1}^{i+k} x_j = 1$ let $H_i^+(x) = \mathbb{R}^2$ and $H_i^-(x) = \emptyset$.

% For $1\leq i\leq 2k$,  define the region $R^i_x$ to be the open wedge bounded by the lines $\ell_{i}(x)$ and $\ell_{i-1}(x)$ and by the arc from $f_{i-1}(x)$ to $f_i(x)$. Let 
% $C^1_x = R^1_x$, and for $2\leq i\leq 2k$ let  $C^i_x$ be the closure of $R^i_x\setminus \bigcup _{j=1}^{i-1}C^j_x$. Note that $C^i_x$ is closed and  convex. 

For $x\in \Delta_{2k-1}$ and  $1\leq i\leq 2k$, define  regions $C_i(x)$ as follows. For $1\leq i\leq k$, let $C_i(x)=(\bigcap_{j=0}^{i-1}H_j^+(x)) \cap H_i^-(x) \cap D$. Similarly, for $k+1\leq i\leq 2k$, let $C_i(x)=(\bigcap_{j=k}^{i-1}H_j^+(x)) \cap H_i^-(x) \cap D$ (see Figure \ref{figure3}). Clearly, $C_i(x)$ is convex for all $i$. We claim that for all $x$, every point in the unit disk is contained in some $C_i(x)$. To see this, let $y\in D$, and assume that $y\in H_0^+(x)$. Let $1\leq i\leq k$ be the smallest index for which $y \in H_i^-(x)$ (such an $i$ exists since $y \in H_k^-(x)$). Then $y \in C_i(x)$. A similar argument holds if $y\in H_0^-(x)= H_k^+(x)$.

We claim that for every $i< j$, the interiors of  $C_i(x)$ and $C_j(x)$ are disjoint.  Indeed, assume first that $1\leq i < j\leq k$. Then $C_j(x) \subseteq H_i^+(x)$ and $C_i(x) \subseteq H_i^-(x)$ and we are done. The case where $k+1\leq i<j \leq 2k$ is similar, and if $1\leq i\leq k$ and $k+1\leq j\leq 2k$, then $C_i(x) \subseteq H_0^+(x)$ and $C_j(x)\subseteq H_0^-(x)$ and we are done.

To invoke the KKM theorem, we also need the fact that if $x_j = 0$ then $\mu_i(C_j(x)) = 0$ for all $i$. Indeed, if $x_j = 0$ then $z_{j-1}(x) = z_j(x)$. It follows that the counterclockwise arc from $z_j(x)$ to $z_{j+k}(x)$ contains both $z_{j-1}(x)$ and $z_{j-1+k}(x)$, and hence $H_j^-(x) \cap H_{j-1}^+(x) \cap D$ is either a point or line. This mean that $C_j(x)$ is contained in a line and the claim follows.

\medskip
{\bf Step 2: Obtaining KKM covers.}
For $i,j\in [2k]$ let $A^i_j\subseteq \Delta_{2k-1}$  be the set of points $x$ such that $\mu_i(C_j(x)) \geq \alpha_j$. Then $A^i_j$ is closed, and by the above claim, if  $x_j=0$ then $x\notin A^i_j$ for all $j$. Additionally, since $\mu_i(\mathbb{R}^2) =1$, we have that $\Delta_{2k-1} \subset \bigcup_{j=1}^{2k} A^i_j$ for all $i$. It follows that the conditions of the colorful KKM theorem  (Theorem \ref{thm:kkm}) are satisfied.

\medskip
{\bf Step 3: Applying the KKM-type theorem.}
By Theorem \ref{thm:kkm} there exists a permutation $\pi:[2k] \rightarrow [2k]$ and a point $x_0\in \bigcap_{i=1}^{2k} A_{\pi(i)}^i$. Taking $Q_i = C_i(x_0)$, the proof is concluded.
\end{proof}

\begin{center}
\begin{figure}
\begin{center}
    \includegraphics[scale=.5]{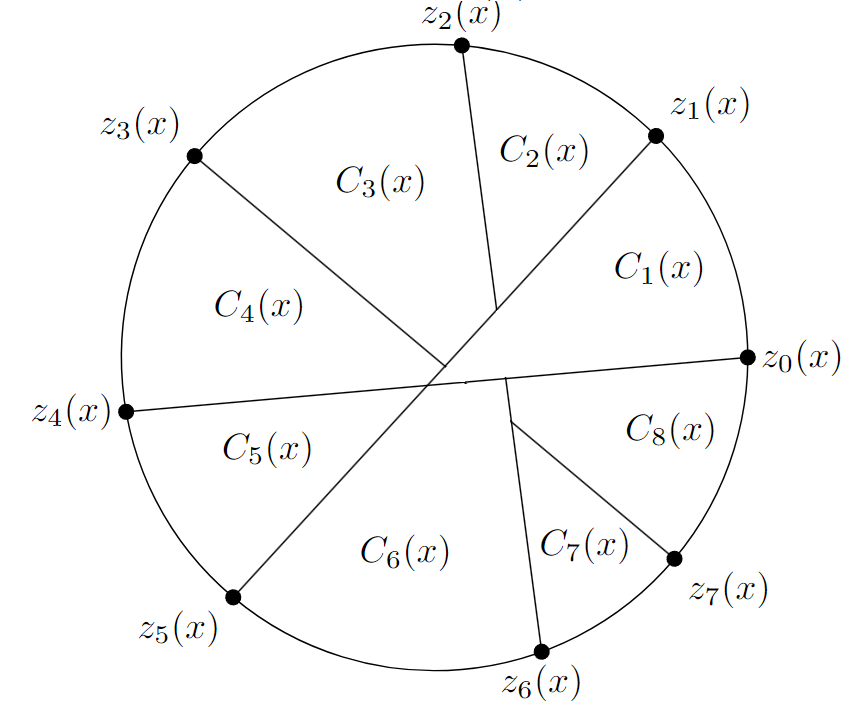}
\end{center}
\caption{The regions $C_i(x)$ in the proof of Theorem \ref{thm:mass}}
    \label{figure3}
\end{figure}
\end{center}

Note that in the above theorem, the convex partition is defined by lines, and moreover, we can also fix one pre-chosen point on one of those lines. Alternatively, we may fix in advance the direction of one of the lines using similar ideas from Section \ref{sec:LinePiercing}. 
This raises the following question: can we use the KKM method to   fix an {\em angle} between two of the lines? An open problem in this direction is the following conjecture of Gr\"unbaum \cite{BG}, which was verified in \cite{AJMR} in certain special cases:

\begin{conjecture}[Gr\"unbaum \cite{BG}]
Let $P$ be a convex body in the plane with area $1$. Then, for
any $t \in [0, 1/4]$ there exists a pair of orthogonal lines that partition $P$ into four
pieces of areas $t, t,(1/2 - t),(1/2 - t)$ in clockwise order.
\end{conjecture}

A problem of Gr\"unbaum \cite{grunbaum1960partitions} asks if there exists equipartitions of a measure by $d$ hyperplane hyperplanes in $\R^d$. The answer to the corresponding question in $\R^d$ for $d\le 3$ is known to be positive, and in fact the case $d=2$ can be proved using the KKM theorem and is a special case of Theorem \ref{thm:mass}. As mentioned above, applying the same approach in higher dimension is more challenging. The case $d=3$ was proven in \cite{hadwiger1966simultane, yao1989partitioning}, and case for $d\geq 5$ was found to be false \cite{avis1984non-partitionable}. This leaves the
following question open.
 \begin{problem}[Gr\"unbaum's measure partition problem \cite{grunbaum1960partitions}]
     Let $\mu$ be a finite absolutely continuous measure in $\R^4$.
Do there
 always exist four hyperplanes that divide $\R^4$
into 16 parts of equal $\mu$-measure?
\end{problem}

A more general problem in this direction is the Gr\"unbaum–Hadwiger–Ramos problem. 

\begin{problem}[The Gr\"unbaum–Hadwiger–Ramos problem]
    Determine all the triples $(d, k, m)$ of positive integers so that the
following  holds: for any $m$ absolutely continuous finite measures in $\R^d$,
there exist $k$  hyperplanes dividing $\R^d$
into $2^k$ parts of equal size in each of the
$m$ measures.
\end{problem} 

The conjectured answer to this problem, due to Ramos \cite{ramos1996equipartition}, is that $(d,k,m)$ is such a triple if $d\geq \left(\frac{2^k-1}{k}\right)m$.
The best known bounds for the problem are obtained in \cite{MVZ}. See also \cite{BFHZ, RPS} for surveys of this problem.

Another potentially interesting question  related to Theorem \ref{thm:mass} is the following. The theorem asserts that there exists a permutation $\pi:[2k]\to [2k]$ and a partition of the plane using $k$ lines into $2k$ convex regions $Q_i$, such that $\mu_{\pi(i)}(Q_i) = \alpha_i$. Suppose we wanted to fix the permutation $\pi$ to be the identity. Can we prove a similar result if we allow the regions not to be convex and the number of lines to be larger? This is the content of the following conjecture, which if true, would generalize a result proved in \cite{CNS} for measures in $\R$ (see also Section \ref{sec:disdiv}).

\begin{conjecture}
There exists a positive integer $r$ with the following property.   Let $\mu_1, \dots, \mu_{k}$ be absolutely
continuous probability measures in the plane, and let $\alpha_1,\dots, \alpha_{k}$ be  positive reals  summing to 1. Then there is a partition of the plane into $2k$ (not necessarily connected) regions $Q_1,\dots, Q_{k}$ using $r$ lines, such that $\mu_i(Q_i) = \alpha_i$ for all $i\in [k]$.
\end{conjecture}

\section{Fair Division}

Fair division is an area of research in the intersection of game theory, economics, and discrete geometry that deals with  partitions of goods among a set of players with subjective preferences. The setting is  as follows. We are given
$n$ players, and a cake (or multiple cakes) to be divided among the players. 
Each cake is identified with the segment $[0,1]$, and a ``piece" is  a sub-segment. 
In every partition of the cake(s), each player has  a list of preferred pieces (or tuple of pieces). 
An envy-free division of the cake(s) is a partition  of the cake(s) into pieces and an allocation of the pieces to (a subset of) the players, such that every player receives a piece (a tuple of pieces) they prefer. Fair division problems can be viewed also as ``colorful" mass partition problems. 
 For  history  and known results see e.g. \cite{Su, NSZ, Soberon} and the references therein.

\subsection{Envy-free division of cakes}

The classical fair division theorem, due to Stromquist \cite{Stromquist} and  Woodall \cite{Woodall}, states the following:
\begin{theorem}[Fair division theorem, Stromquist \cite{Stromquist},  Woodall \cite{Woodall}, 1980]\label{thm:cfd}
Suppose we have  a cake and $n$ players,
and in every partition of the cake into $n$ pieces every player has a list of their preferred pieces. Suppose further that the following two conditions hold: 
\begin{enumerate}
\item[(1)] {\em The players are hungry}: in every partition  of the cake into $n$ pieces every player prefers at least one positive-length piece.
\item[(2)] {\em The preference sets are closed}: if a player prefers piece $i$ in a converging sequence of partitions, then they prefer piece $i$ also in the limit partition.
\end{enumerate}
Then  there exists an envy-free division of the cake. 
\end{theorem}

In \cite{Su}, Su gave a proof of this theorem using Sperner's lemma. 
Note however that Theorem \ref{thm:cfd} is the colorful KKM theorem in disguise. Indeed, 
the set of all partitions of the cake can be modeled using the  $(n-1)$-simplex $\Delta_{n-1}$, where a point $x=(x_1,\dots,x_n) \in \Delta_{n-1}$ corresponds to the partition in which the $i$-th piece (ordered from left to right) is of length $x_i$. For $i,j\in [n]$ define sets $A^j_i = \{x\in \Delta_{n-1} \mid \text{player 
 } j \text{ prefers piece } i \text{ in the partition }x\}.$
 Then condition (1) of the theorem is equivalent to the condition that the sets $\{A^j_1,\dots, A^j_n\}$ satisfy the KKM covering condition for every $j$, and condition (2) is equivalent to the condition that the sets $A^j_i$ are closed. The conclusion of the colorful KKM theorem, that there exists a permutation $\pi:[n]\to [n]$ and a point $x\in \bigcap A^j_{\pi(j)}$, is equivalent to $x$ being an envy-free division of the cake, where player $j$ is allocated the piece $\pi(j)$. 

Woodall \cite{Woodall} additionally showed that Theorem \ref{thm:cfd} holds even if the preferences of one of the players is not known, or ``secretive''. Specifically, there exists a partition of the cake such that no matter which piece the secretive player chooses as their preferred piece, there is an allocation of the remaining pieces to the rest of the players that is envy-free. Another proof of this result was given in \cite{asada2018}. The KKM-type theorem corresponding to the secretive player version is the following:

\begin{theorem}[Woodall 1980 \cite{Woodall}, Asada et al. 2018 \cite{asada2018}]\label{thm:secretKKM}
For every $j\in [k]$, let $\{A_1^j,\dots,A_{k}^j\}$ be a KKM cover of $\Delta_{k-1}$. Then there exists $x\in \Delta_{k-1}$ such that for each $i\in[k]$ there exists a permutation $\pi_i:  [k-1] \rightarrow [k]\setminus \{i\}$ satisfying $x\in \bigcap_{j=1}^{k-1} A^j_{\pi_i(j)}$.
\end{theorem}

A more general version of Theorem \ref{thm:secretKKM} appears in \cite{panina2022envy}.

\medskip
In \cite{NSZ}, a generalization of the classical envy-free division theorem for  multi-pieces was proved. The proof used the KKM method, together with the KKMS theorem (Theorem \ref{kkms}) and its polytopal extension due to Komiya \cite{Komiya}.  

\begin{theorem}[Nyman, Su, Zerbib 2020 \cite {NSZ}]\label{thm:multcakes} \hfill
\begin{enumerate}
\item  Suppose a cake and $n$ players satisfy the following:
(a) in any partition of cake into n pieces, each player prefers some $k$-tuple of 
positive length pieces, and
(b) the preference sets are closed.
Then there is a partition of the cake in which at least $\lceil\frac{n}{k^2-k+1}\rceil$ players prefer pairwise disjoint $k$-tuple of pieces.  
\item Suppose $k$ cakes and $p=k(n - 1) +
1$ players satisfy the following:
(a) in any partition of cakes into $n$ pieces each, each player prefers some $k$-tuple of 
positive length pieces, one piece from each cake, and
(b) the preference sets are closed.
Then there is a partition of the cakes in which at least $\lceil\frac{p}{k^2-k}\rceil$ players prefer pairwise disjoint $k$-tuple of pieces.  
\end{enumerate}

\end{theorem}

 Additional topological techniques, such as the topological Hall theorem \cite{AH} and the ``Meshulam game", were applied in \cite{ABBSZ} for step 4 of the KKM method, to prove  further bounds on fair division with multiple cakes, improving the bounds given in Theorem \ref{thm:multcakes}. 
\medskip

A natural question is whether the hungry player condition is necessary.  
In \cite{MZ} it was shown that if the number of players $n$ is a prime number or $n=4$ then it is possible to find an envy-free division of the cake  even if the hungry players condition is not satisfied (another proof was later given in \cite{panina2023configuration}). This verified a conjecture of  Segal-Halevi \cite{SegalHalevi}.
The main step in our proof was a new topological lemma, reminiscent of Sperner's lemma. Later this result prime power $n$, and was disproved for all $n$ that is not a prime power \cite{AK}.

In a slightly different setting, Sober\'on \cite{Soberon} proved  a ``sparse" version of the classical fair division theorem. In his theorem not all the players are hungry, but  at least one player in every set of $n-k+1$ players is hungry. The conclusion is that there is a envy-free division of the cake to $k$ players. To this end, he proved a sparse-colorful version of the KKM theorem, and conjectured that there is a polytopal extension of this theorem. This conjecture was proved  in \cite{MZ2} (Theorem \ref{thm:sparse-komiya}). 

Sober\'on and Yu \cite{soberon2023} proved a higher dimensional version of Theorem \ref{thm:cfd} for nested hyperplane partitions.
Here the measures replace the role of the players as in Theorem \ref{thm:cfd}. Given measures $\mu_1,\dots,\mu_n$, we say that a convex partition $(C_1,\dots,C_n)$ of $\mathbb{R}^d$ can be allocated to the measures in an envy-free way if there exists a permutation $\pi$ of $[n]$ such that $\mu_i(C_{\pi(i)})\geq \mu_i(C_{\pi(i')})$ for all $i,i' \in [n]$.

\begin{theorem}[Sober\'on - Yu 2023 \cite{soberon2023}]
     Let $\mu_1,\dots,\mu_n$ be absolutely continuous probability measures in $\R^d$ and let $v_1,\dots,v_{n-1} \in \R^d$. Then there is a  nested hyperplane partition $(c_1,\dots,C_n)$ of $\mathbb{R}^d$  with respect to the vectors $v_1,\dots,v_{n-1}$, such that there is an envy-free allocation of the regions to the measures. 
\end{theorem}

\subsection{Dual fair division theorems}

The classical envy-free division theorem of Stromquist and Woodall guarantees an envy-free partition of \textit{desirables} among $n$ players. Dual to this notion is assigning \textit{undesirables} to $n$ players in an envy-free manner. For example, if $n$ players are to divide among themselves the total cost of rent for $n$ rooms, we would like to assign to each player a unique room, and a percentage of the total rent they must pay for that room, in an envy-free manner. Su \cite{Su} proved that under certain conditions such an envy-free rent division does exist:

\begin{theorem}[Rental harmony theorem, Su 1999 \cite{Su}]\label{thm:rentDiv}
Suppose $n$ players in a house with $n$ bedrooms would like to decide who gets which room and for what percentage of the rent. Suppose further that the following is satisfied:
\begin{enumerate}
    \item In every assignment of rent to the rooms, each player  finds some room acceptable.
    \item The players  always prefer a free room to a non-free room.
    \item The preference sets are closed.
\end{enumerate}
Then there exists an assignment of rent to each room, and an allocation of the rooms to the players, so that each player prefers their room with its assigned rent.
\end{theorem}

Similarly to Theorem \ref{thm:secretKKM}, there is a version of Theorem \ref{thm:rentDiv} where the preferences of one of the players is secret \cite{fricksec}.

As the fair division theorem corresponds to the colorful KKM theorem, Theorem \ref{thm:rentDiv} corresponds a colorful version of Theorem \ref{thm:cfd}. 
In \cite{NSZ} a multiple house version was proved. 
\begin{theorem}
    Under similar conditions to those of Theorem \ref{thm:rentDiv}, if $2n - 1$ players seek to rent two rooms, one
in each of two buildings containing $n$ rooms each, then there exists a
division of rents in which a subset of $n$ players each get their preferred
two rooms.
\end{theorem}
Further results of this nature can be found in \cite{NSZ, ABBSZ}

% \begin{theorem}[\textcolor{blue}{Is there a citation for this?}]
% Suppose that subsets $B^j_i \subset \Delta_{n-1}$ for $j,i \in [n]$ are either all open or all closed and satisfy the following conditions for all $j\in [n]$:
% \begin{enumerate}
%     \item $\bigcup_{i=1}^n B_i^j = \Delta_{n-1}$,
%     \item For a point $x=(x_1,\dots,x_n) \in \Delta_{n-1}$, $x_i = 0 \implies x\in B_i^j$.
% \end{enumerate}
%     Then there exists a permutation $\pi$ of $[n]$ such that $\bigcap_{i=1}^n B_{\pi(i)}^i \neq \emptyset$.
% \end{theorem}

\subsection{Disproportionate division}\label{sec:disdiv}

Fair division is fairly well-understood when we wish to divide the goods ``evenly" between the players. However, the problem  is not yet completely solved if the distribution we look for is disproportionate. More precisely, given $n$  measures $\mu_i, i\in [n],$ on $[0,1]$ and $n$ positive reals $\alpha_i, i\in [n],$  summing to 1, one wishes to find a partition of $[0,1]$ into $n$ parts $Q_1,\dots,Q_n$ (not necessarily intervals), such that $\mu_i(Q_i) = \alpha_i$.
This problem was first introduced by Segal-Halevi in \cite{SHdis}, who showed that if the numbers $\alpha_i$ are not all equal to $1/n$, then the parts $Q_i$ cannot be chosen to be intervals in general. This leads to the following question:
\begin{problem}[Segal-Halevi 2018 \cite{SHdis}]\label{prob:dis}
    Find the minimum number $q$ such that for every $n$ probability measures $\mu_1,\dots,\mu_n$ on $[0,1]$, and every positive reals $\alpha_1,\dots, \alpha_n$ so that $\sum_{i=1}^n\alpha_i=1$, there exists a partition of $[0,1]$ into $n$ (not necessarily convex) parts $Q_1,\dots,Q_n$ using at most $q$ cuts, such that $\mu_i(Q_i) = \alpha_i$.
\end{problem}
In \cite{CNS} it was proved that $q\le 3n-4$. The best current lower bound, found in \cite{SHdis}, is $q\ge 2n-2$.

The KKM method can solve Problem \ref{prob:dis} in the case that $\alpha_i=1/n$ for all $i$ (indeed, this is a special case of the fair division theorem). However, this case follows  by the simple greedy argument. Slide a knife starting at 0  to the right until we reach the first point $x$ for which there exists an $i$ where $\mu_i([0,x]) = 1/n$. Then let  $Q_i=[0,x]$. We now have that $\mu_j([x,1])\geq \frac{n-1}{n}$ for all $j\neq i$, so we can repeat the argument for the remaining $n-1$ measures on the interval $[x,1]$. The same argument can be used to prove the following.

\begin{theorem}
    For every $n$ probability measures $\mu_1,\dots,\mu_n$ on $[0,1]$ and  positive reals $\alpha_1,\dots, \alpha_n$ so that $\sum_{i=1}^n\alpha_i=1$, there exists a permutation $\pi:[n]\to[n]$ and a partition of $[0,1]$ into $n$ intervals $Q_1,\dots,Q_n$, such that $\mu_i(Q_{\pi(i)}) = \alpha_i$.
\end{theorem}

As mentioned in \cite{CNS}, the following conjecture, which appears to be topological in nature, will imply that $q=2n-2$ in Problem \ref{prob:dis}.

\begin{conjecture}[Crew, Narayanan, Spirkl 2019 \cite{CNS}]
    For any $n$ probability measures $\mu_1,\dots,\mu_n$
on the unit circle $S^1$
 and non-negative reals $\alpha_1,\dots, \alpha_n$ so that $\sum_{i=1}^n\alpha_i=1$, there
exists a partition  $[n] = P \cup Q$ into two nonempty sets $P,Q$ and a partition of the circle $S^1 = X \cup X^c$
into two intervals such that $\min_{i\in P} \mu_i(X)= \sum_{j\in P} \alpha_j$ and $\min_{i\in Q} \mu_i(X^c)= \sum_{j\in Q} \alpha_j$.
\end{conjecture}

\section{Matching theory}

In this section we mention a different set of results that were proven using the KKM theorem. Those results concern Topological extensions of Hall's  theorem \cite{hall}.
A main question in graph and hypergraph theory  is to find sufficient (and sometimes also necessary) conditions for the existence of a matching in certain families of graphs or hypergraphs. Hall's marriage theorem  provides such a  condition for bipartite graphs on sides $A$ and $B$: there exists a matching saturating $A$ if and only if for every subset of vertices $X\subseteq A$, the size of the neighborhood of $X$ is at least as large as the size of $X$.

One generalization of Hall's theorem follows from Edmonds' two matroids intersection theorem \cite{edmonds1979}. For terminology and definitions in matroid theory, see \cite{oxley}. Denote by $\rho_\mathcal{M}$ (or  $\rho$ if $\mathcal{M}$ is understood from the context)  the rank function of $\mathcal{M}$. For a subset $X\subset V$, let $\mathcal{M}.X$  denote the matroid whose independent sets are the subsets $S\subset X$ for which $S \cup T$ is independent for every independent set $T \subset V\setminus X$. For two matroids $\mathcal{M}$ and $\mathcal{N}$  on the same ground set $V$,  we say that $\mathcal{M}$  {\em is  matchable to} $\mathcal{N}$ if there is a base $B$ of $\mathcal{M}$ that is independent in $\mathcal{N}$. The following consequence of Edmonds' two matroids intersection theorem is stated for instance in \cite{aharoni2006intersection}.

\begin{theorem}\label{thm:Edmonds}
Let $\mathcal{M},\mathcal{N}$ be matroids on the same ground set $V$. Then $\mathcal{M}$  is  matchable to $\mathcal{N}$ if and only if $\rho_\mathcal{N}(X) \geq \rho(\mathcal{M}.X)$ for every $X\subset V$.
\end{theorem}

To see that Theorem \ref{thm:Edmonds} implies Hall's theorem, let $G$ be a bipartite graph with sides $A, B$ that satisfies Hall's condition, i.e., for all $X\subset A$, $|N(X)| \geq |X|$. Define two matroids on the same ground set $V=E(G)$ as follows. Let $\mathcal{N}$  be the matroid whose independent sets are the set of edges of $G$ whose endpoints in $A$ are disjoint, and  let $\M$  be the matroid whose independents sets are the sets of edges of $G$ whose endpoints in $B$ are disjoint. For a subset of edges $X\subseteq E(G)$, let $Y(X)\subseteq A$ be the set of endpoints of $X$ in $A$. Then $\rho_{\mathcal{N}}(X) = |Y(X)|$ and $\rho(\M.X) \leq |N(Y(X))|$. Thus, the hypothesis of Theorem \ref{thm:Edmonds} is satisfied, so $\mathcal{M}$  is  matchable to $\mathcal{N}$, meaning that $G$ has a matching saturating $A$.

 Aharoni and Berger \cite{aharoni2006intersection} generalized Theorem \ref{thm:Edmonds} by replacing $\mathcal{N}$ with a simplicial complex $\mathcal{C}$. Their proof is a beautiful application of the KKM theorem.

A simplicial complex  $\mathcal{C}$ on a ground set $V$ is a collection of subsets of $V$ satisfying that if $\sigma \in  \mathcal{C}$ and $\tau\subset \sigma$, then $\tau \in  \mathcal{C}$.
For definitions concerning simplicial complexes and connectivity see for instance \cite{matousek2003using}. 
Given a  simplicial complex $\mathcal{C}$ with ground set $V$, let $\eta(\mathcal{C})$ denote the connectivity of $\mathcal{C}$ plus 2. For a subset $X\subset V$, denote by $\mathcal{C} | X$  the subcomplex induced by $X$, i.e the simplicial complex consisting of the sets $\sigma\in \mathcal{C}$ such that $\sigma\subset X$. We say that a matroid $\mathcal{M}$ and a simplicial complex $\mathcal{C}$ on the same ground set $V$ are {\em matchable} if there is a base of $\mathcal{M}$ which is a simplex of $\mathcal{C}$.

\begin{theorem}[Aharoni - Berger 2006 \cite{aharoni2006intersection}]\label{thm:ABMatchable}
    Let $\mathcal{M}$ be a matroid and $\mathcal{C}$  a simplicial complex, both on the same ground set $V$. If $\eta(\mathcal{C} | X) \geq \rho(\mathcal{M}.X)$ for all $X\subset V$, then $\mathcal{M}$ and $\mathcal{C}$ are matchable.
\end{theorem}

In all the  applications of the KKM theorem we mentioned in previous sections, the KKM theorem was used to find a particular division of an ambient Euclidean space. The application of the KKM theorem in the proof of Theorem \ref{thm:ABMatchable} is different, as  there is no ambient Euclidean space inherently present in the theorem, and  the proof does not involve the idea of associating the points of a simplex with a division of some space.
We will briefly sketch the proof here.  For full details see \cite{aharoni2006intersection}.

Let $\M$ be a matroid and let $\mathcal{C}$ be a simplicial complex, both on the same ground set $V$. 
 A set $F\subset V$ is a \textit{flat} of $\M$ if  for every $x\in V\setminus F$, $\rho_{\M}(F\cup \{x\}) > \rho_{\M}(F)$. The {\em flat complex} of $\M$, denoted  $\F(\M)$, is the simplicial complex whose vertices are the flats of $\M$ and whose simplices are chains of flats $F_1\subset \cdots\subset F_m$. 
 The \textit{cocircuits} of $\M$ are  the circuits (minimal dependent sets) of the dual of $\M$, which is matroid whose bases are the complements of the bases of $\M$. 
 Let $||\mathcal{C}||$ denote the geometric realization of a simplicial complex $\mathcal{C}$. For a vertex $x\in V$, the support $\supp_{\mathcal{C}}(x)$ is the minimal face of $\mathcal{C}$ containing $x$.
\medskip

\noindent {\em Sketch of the proof of Theorem \ref{thm:ABMatchable}.}
 First, a continuous map $\xi: ||\F(\M)|| \rightarrow ||\mathcal{C}||$ is constructed by successively extending the map from the $(i-1)$-skeleton to the $i$-skeleton of $\F(\M)$, making use of the lower bound on the connectivity of induced subcomplexes of $\mathcal{C}$ provided by the hypothesis of Theorem \ref{thm:ABMatchable}.
 The map $\xi$ has the additional property that for all $x\in ||\F(\M)||$, $\supp_{\F(\M)}(x)$  contains a flat $F$ which is disjoint from $\supp_\mathcal{C}(\xi(x))$. 

Then a simplex $\Delta$ is introduced, where  the vertices  of $\Delta$ are indexed by the cocircuits of $\M$. In other words, for every cocircuit $D$ of $\M$, there is a corresponding vertex $v_D$ of $\Delta$. A  map $\pi:bs(\Delta) \rightarrow \F(\M)$ from the   barycentric subdivision $bs(\Delta)$ of $\Delta$ to $\F(\M)$  is then defined in natural way. 

The KKM theorem is then applied to $\Delta$. For a cocircuit $D$,  define $A_D$ to be the set of points $x$  in $\Delta$ such that the support of $\xi(\pi(x))$ in $\mathcal{C}$ has a non-empty intersection with $D$. The properties of the maps $\xi$ and $\pi$ imply that the collection $\{A_D\mid D \text{ a cocircuit of }\M\}$ is a KKM cover. The conclusion of the KKM theorem then entails  that $\bigcap_D A_D \neq \emptyset$. Therefore, for a point $y \in \bigcap_D A_D$, we have that the support of $\xi(\pi(y))$ in $\mathcal{C}$ has a nonempty intersection with each cocircuit of $\M$. However, such a set must have full rank in $\M$ and hence contains a basis of $\M$ that is simultaneously a simplex of $\mathcal{C}$, which completes the proof. \qed

\end{document}